\documentclass[reqno,15pt]{amsart}
\usepackage{amsmath}
\usepackage{amssymb}
\usepackage{amscd}
\usepackage{graphics}
\usepackage{latexsym}
\usepackage{hyperref}

\theoremstyle{plain}
\newtheorem{theorem}{Theorem}[section]
\newtheorem{thm}[theorem]{Theorem}
\newtheorem{cor}[theorem]{Corollary}
\newtheorem{lem}[theorem]{Lemma}
\newtheorem{prop}[theorem]{Proposition}

\newcounter{kludge}
\newcounter{kludgeb}
\theoremstyle{definition}
\newtheorem{defn}[theorem]{Definition}

\newtheorem{prob}[theorem]{Problem}

\newtheorem{rmk}[theorem]{Remark}

\theoremstyle{remark}

\input xy
\xyoption{all}

\newcommand{\marpar}[1]{}
\newcommand{\mni}{\medskip\noindent}

\newcommand{\mbb}{\mathbb}
\newcommand{\QQ}{\mbb{Q}}

\newcommand{\ZZ}{\mbb{Z}}

\newcommand{\PP}{\mbb{P}}

\newcommand{\mc}{\mathcal}

\newcommand{\mcL}{\mc{L}}

\newcommand{\mf}{\mathfrak}

\newcommand{\OO}{\mc{O}}

\newcommand{\wht}{\widehat}

\newcommand{\wh}{\widehat}

\newcommand{\wt}{\widetilde}

\newcommand{\ol}{\overline}
\newcommand{\ul}{\underline}

\newcommand{\SP}{\text{Spec }}

\newcommand{\M}{\overline{\mc{M}}}

\newcommand{\m}{\overline{\text{M}}}

\newcommand{\kgnb}[1]{\m_{#1}}
\newcommand{\Kgnb}[1]{\M_{#1}}

\newcommand{\Gm}[1]{\mathbb{G}_{#1}}

\newcommand{\Xx}{X}

\newcommand{\Cc}{C}

\newcommand{\fF}{\kappa}

\newcommand{\lt}{\left}
\newcommand{\rt}{\right}

\newsavebox{\sembox}
\newlength{\semwidth}
\newlength{\boxwidth}

\newsavebox{\semrbox}
\newlength{\semrwidth}
\newlength{\boxrwidth}

\title
{Weak Approximation for Fano
  Complete Intersections in Positive Characteristic}  

\author[Starr]{Jason Michael Starr}
\address{Department of Mathematics \\
  Stony Brook University \\ Stony Brook, NY 11794 USA}
\email{jstarr@math.sunysb.edu} 

\author[Tian]{Zhiyu Tian}
\address{
Beijing International Center for Mathematical Research 77103 \\
Peking University \\
No.5 Yiheyuan Road \\
Haidian District, Beijing \\ 
China}
\email{zhiyutian@gmail.com}

\author[Zong]{Runhong Zong}
\address{
  Institut f\"{u}r Mathematik \\
  Johannes Guttenberg--Universit\"{a}t Mainz \\
  Raum 04-131 \\
  Staudingerweg 9 \\
  55128 Mainz \\
  Deutschland
}
\email{zong@uni-mainz.de}

\date{\today}

\begin{document}


\begin{abstract} 
  For a smooth curve $B$ over a field $k=\ol{k}$ with
  $\text{char}(k)=p$,  
  for
  every $B$-flat complete intersection $X_B$ in
  $B\times_{\SP k} \PP^n_k$ of type $(d_1,\dots,d_c)$, we prove weak
  approximation of $\wh{\OO}_{B,b}$-points of $X_B$ by $k(B)$-points
  at all places $b$ of (strong) potentially good reduction,
  if the Fano index is $\geq 2$ and if
  $p>\max(d_1,\dots,d_c)$. Assuming this inequality, which
  is nearly sharp, such complete intersections are \emph{separably
    uniruled by lines}, and even separably rationally connected.  This
  also applies to specializations of complex Fano manifolds with
  Picard rank $1$ and Fano index $1$ away from ``bad primes''.
\end{abstract}


\maketitle



\section{Introduction and Statement of Results} \label{sec-int} 
\marpar{sec-int}

\mni
The formulation of weak approximation is close to the infinitesimal
lifting property for smooth morphisms.  So we review this; the expert
can skip to Theorem \ref{thm-WA}.

\mni
For each integer $n\geq 0$, the functor of \textbf{projective
  $n$-space}, $\PP^n_R$, over a (commutative, unital) ring $R$
associates to every $R$-algebra $A$ all equivalence classes of
$A$-module quotients of $A^{\oplus (n+1)}$ that are locally free of
rank $1$.  These are locally equivalent to $A^\times$-orbits (under
simultaneous scaling) of $(n+1)$-tuples
$(a_0,\dots,a_n)\in A^{\oplus(n+1)}$ such that the $A$-submodule
$\langle a_0,\dots,a_n\rangle \subset A$ equals $A$.

\mni
A \textbf{projective $R$-scheme in} $\PP^n_{R}$ is a zero locus
$\Xx_R = \text{Zero}(g_1,\dots,g_c)$ for a $c$-tuple $(g_1,\dots,g_c)$
of homogeneous polynomials $g_j\in R[t_0,\dots,t_n]$ of degree
$\text{deg}(g_j)=d_j>0$.  An $A^\times$-orbit $[a_0,\dots,a_n]$ is in
$\Xx_R$ if and only if $g_j(a_0,\dots,a_n)$ equals $0$ for every
$j=1,\dots,c$.  For $S$ an $R$-algebra, the \textbf{base change}
scheme $\Xx_S\subset \PP^n_S$ represents the restriction of this
functor to $S$-algebras.

\mni
For $c\leq n$, the \textbf{Jacobian ideal} of the projective
$R$-scheme $\Xx_R=\text{Zero}(g_1,\dots,g_c)$ is the homogeneous ideal
in in $R[t_0,\dots,t_n]/\langle g_1,\dots,g_c\rangle$ generated by all
$c\times c$ minors of the $c\times (n+1)$ Jacobian matrix
$[\partial g_j/\partial t_i]$.  By the Jacobian criterion, the
$R$-\textbf{smooth locus of relative dimension} $n-c$ is the open
complement in $\Xx_R$ of the zero locus of the Jacobian ideal.  If
this equals all of $\Xx_R$, then $\Xx_R$ is an \textbf{$R$-smooth
  complete intersection}.  In this case, it is also $R$-\textbf{flat
  of relative dimension} $n-c$: for every $R$-algebra $\fF$ that is a
field, the $\fF$-scheme $\Xx_\fF$ has dimension $n-c$.

\mni
For a general $R$-scheme $\Xx_R$, the $R$-\textbf{smooth locus of
  relative dimension} $n-c$ is the union of all open subschemes of
$\Xx_R$ that are $R$-isomorphic to an open subscheme of the $R$-smooth
locus of relative dimension $n-c$ for $\text{Zero}(g_1,\dots,g_c)$ as
above.  If this equals all of $\Xx_R$, then $\Xx_R$ is
$R$-\textbf{smooth of relative dimension} $n-c$.


\subsection{Weak Approximation over Function Fields} 
\label{subsec-WA} \marpar{subsec-WA}

\mni
A quasi-projective $R$-scheme $\Xx_R$ is $R$-smooth of some (locally
constant) relative dimension precisely if the following map is
surjective for every $e\geq 0$ and for every $R$-algebra $\wht{\OO}$
that is a complete, local Noetherian ring with unique maximal ideal
$\wh{\mf{m}}$,
$$
\rho_{\Xx_R,\wht{\OO},e}:\Xx_R(\wh{\OO}) \to \Xx_R(\wh{\OO}/\wh{\mf{m}}^{e+1}).
$$

\mni
The coordinate ring $R=\OO_B$ of a connected, affine Dedekind scheme
$B$ is a normal, Noetherian, integral domain with maximal ideal
$\mf{m}_{B,b}$ for every closed point $b\in B$.  The associated local
ring $\OO_{B,b}$ is a \textbf{discrete valuation ring} (DVR):
$\mf{m}_{B,b}\subset \OO_{B,b}$ is a free module of rank $1$.  Denote
by $\wh{\OO}_{B,b}$, resp. by $\wh{\mf{m}}_{B,b}$, the
$\mf{m}_{B,b}$-completion
$\varprojlim_e \OO_{B,b}/\mf{m}_{B,b}^{e+1}$, resp. the completion of
$\mf{m}_{B,b}$.  Denote the restriction map by
$$
\rho_{\Xx_B,b,e}:\Xx_B(\wh{\OO}_{B,b})\to
\Xx_B(\wh{\OO}_{B,b}/\wh{\mf{m}}_{B,b}^{e+1}).
$$
A projective $B$-scheme $\Xx_B$ satisfies \textbf{weak approximation}
at distinct closed points $(b_1,\dots,b_\delta)$ of $B$ if for every
integer $e\geq 0$ the following two maps have equal image,
$$
\rho_{\Xx_B,b_1,\dots,b_\delta,e}:
\Xx_B(B) \to
\prod_{i=1}^\delta \Xx_B(\wh{\OO}_{B,b_i}/\wh{\mf{m}}_{B,b_i}^{e+1}),
$$
$$
\prod_{i=1}^\delta\rho_{\Xx_B,b_i,e}
:
\prod_{i=1}^\delta \Xx_B(\wh{\OO}_{B,b_i})
\to
\prod_{i=1}^\delta
\Xx_B(\wh{\OO}_{B,b_i}/\wh{\mf{m}}_{B,b_i}^{e+1}).
$$
When $\Xx_B$ is $B$-smooth over an open $U\subset B$ containing all
$b_i$ then each $b_i$ is a \textbf{place of good reduction}, in which
case weak approximation is equivalent to surjectivity of the first map
$\rho_{\Xx_B,b_1,\dots,b_\delta,e}$ for every integer $e\geq 0$.

\mni
Similarly, a closed point $b$ of $B$ is a \textbf{place of potentially
  good reduction as a complete intersection} if there exists a local
homomorphism of DVRs,
$$
(\wh{\OO}_{B,b},\wh{\mf{m}}_{B,b})\to (\wh{\OO}',\wh{\mf{m}}'),
$$
and an $\wh{\OO}'$-smooth complete intersection $\Xx'_{\wh{\OO}'}$
with tame, Galois extension,
$$
\Gamma = \text{Aut}(\text{Frac}(\wh{\OO}')/\text{Frac}(\wh{\OO}_{B,b})),
$$
such that $\Xx'_{\text{Frac}(\wh{\OO}')}$ is
$\text{Frac}(\wh{\OO}')$-isomorphic to the base change
$\Xx_{\text{Frac}(\wh{\OO}')}$ of the generic fiber
$\Xx_{\text{Frac}(R)}$.  If also the natural $\Gamma$-semilinear
structure on $\Xx_{\text{Frac}(\wh{\OO}')}$ extends to a
$\Gamma$-semilinear structure on $\Xx'_{\wh{\OO}'}$, then $b$ is a
\textbf{place of (strong) potentially good reduction as a complete
  intersection}.  When $\kappa(b)$ is algebraically closed of
characteristic prime to $\#\Gamma$, then the extension is
\textbf{tame}.

\begin{thm}[Weak Approximation for Complete Intersections] \label{thm-WA} \marpar{thm-WA}
  For every smooth, affine, connected curve $B$ over a field
  $k=\ol{k}$, for every $B$-flat complete intersection
  $\Xx_B\subset B\times_{\SP k} \PP^n_k$ of $c$ hypersurfaces of
  degrees $(d_1,\dots,d_c)$, weak approximation holds at every place
  of (strong) potentially good reduction as complete
  intersections provided that the Fano index is $\geq 2$, and
  $p:=\text{char}(k)$  satisfies $p>\max(d_1,\dots,d_c)$.
\end{thm}

\mni
The Fano index is $i:=n+1-(d_1+\dots+d_c)$.  Please compare this
theorem to Tsen's Theorem below.  Tsen's Theorem is sharp: whenever
$i\leq 0$, there are examples with $\Xx_B(B)$ empty.  Also, there are
many examples of \emph{singular} $\Xx_B$ with $i$ positive such that
$\Xx_B(B)$ is a singleton set.

\begin{thm}[Tsen's Theorem] \cite{Tsen36} \label{thm-Tsen} \marpar{lem-Tsen}
  With notation as above, there exists at least one $B$-point of
  $\Xx_B$ provided that the Fano index is $\geq 1$.
\end{thm}

\mni
In characteristic $0$, the \textbf{Weak Approximation Conjecture} of
Hassett and Tschinkel predicts that weak approximation holds at all
places for every $B$-flat projective variety $\Xx_B$ whose geometric
generic fiber $\Xx_{\ol{\text{Frac}(B)}}$ is smooth and
\emph{separably rationally connected}.  The Weak Approximation
Conjecture in characteristic $0$ was proved by Hassett and Tschinkel
at places of good reduction, \cite{HT06}, and it was proved by the
second and third authors at places of (strong) potentially good
reduction, \cite{TZWA}.  In particular, this proves Theorem
\ref{thm-WA} in characteristic $0$.

\mni
The proofs of all of these theorems, including Theorem \ref{thm-WA},
make essential use of the \emph{infinitesimal deformation theory} over
complete local rings such as $\wht{\OO}_{B,b}$ for closed subschemes
of a projective scheme, as developed by Kodaira-Spencer, Mike Artin,
Illusie, et al.  They also make essential use of \emph{rational
  curves} in geometric fibers of $\Xx_R$ (images of non-constant
morphisms with domain $\PP^1$), particularly those \emph{free rational
  curves}, resp. \emph{very free rational curves}, such that
infinitesimal deformations of the morphism freely extend even after
precomposing by a positive-degree self-map of $\PP^1$, resp. such that
for each effective divisor in $\PP^1$, after precomposing by a
self-map of $\PP^1$ of sufficiently positive degree, infinitesimal
deformations of the morphism relative to the divisor extend freely.
Smooth projective schemes that contain such a rational curve are
\emph{separably uniruled}, resp. \emph{separably rationally
  connected}; this is intimately connected to positivity properties of
the tangent bundle $T_{\Xx}$.


\subsection{Separable Rational Connectedness and Weak Approximation} 
\label{subsec-RC} \marpar{subsec-RC}

\mni
A smooth projective variety $\Xx$ is \textbf{Fano} if the first Chern
class $c_1(T_{\Xx})$ of the tangent bundle $T_{\Xx}$ is ample, in
which case the \textbf{Fano index} is the largest positive integer $i$
such that $c_1(T_{\Xx})$ equals $i$ times an (integral) Cartier
divisor class.  For smooth complete intersections in $\PP^n$ of
dimension $n-c\geq 3$, the \emph{Picard group is cyclic} (every
projective embedding is obtained by a Veronese re-embedding followed
by a linear projection), so that the entire Picard group is generated
by a unique ample generator $\OO(1)$, whose Cartier divisor class is
called the \emph{hyperplane class}.  By the adjunction formula, the
first Chern class is $i$ times the hyperplane class for
$$
i = (n+1)-(d_1+\dots+d_c).
$$
Thus, the complete intersection is Fano if and only if $i\geq 1$, in
which case the Fano index equals $i$.  In particular, every curve in
$\Xx$ has $c_1(T_{\Xx})$-degree $\geq i$.

\mni
The study of rational curves on varieties has been crucial to the
study of Fano manifolds since Mori's seminal paper proving that every
Fano manifold is covered by rational curves of $c_1(T_{\Xx})$-degree
bounded by $\text{dim}(\Xx)+1$, \cite{Mori}.  In particular,
$\text{dim}(\Xx)+1$ is an upper bound for the least
$c_1(T_{\Xx})$-degree of a rational curve in $\Xx$.  This least
$c_1(T_{\Xx})$-degree is the \textbf{Fano pseudoindex} $\psi_{\Xx}$ of
$\Xx$.  For smooth complete intersections of dimension $n-c\geq 3$,
the Picard group of $\Xx$ is cyclic, and every Fano complete
intersection contains lines so that the Fano pseudoindex $\psi_{\Xx}$
equals the Fano index $i$.

\mni
Fano manifolds of Fano pseudoindex $i=1$ are special: there are entire
irreducible components of the parameter space representing
$\fF$-morphisms from $\PP^1_{A}$ to $X$ that are everywhere
\textbf{obstructed} (they have nonzero obstruction group to
infinitesimal deformations as constructed by Kodaira-Spencer, Mike
Artin, Illusie, et al.), and this persists even after precomposing
with a positive degree self-map of $\PP^1$ and allowing small
deformations.  To avoid this pathology, we must restrict to Fano
manifolds of Fano pseudoindex $\geq 2$.

\mni
A non-constant $\fF$-morphism from $\PP^1_{\fF}$ to an $\fF$-smooth,
projective variety $\Xx$ is \textbf{free} if the pullback of $T_{\Xx}$
is globally generated on $\PP^1$: this is equivalent to
unobstructedness of every morphism obtained by precomposing with a
positive degree self-map $\PP^1_{A}\to \PP^1_{A}$, and it is stable
under small deformations of the morphism.  A smooth projective variety
is \textbf{separably uniruled} if and only if it has a free rational
curve, this implies that $\Xx$ is covered by curves with positive
$c_1(T_X)$-degree, and these are equivalent in characteristic $0$,
\cite{Miyaoka-Mori86}.  Similarly an $\fF$-morphism from $\PP^1_{\fF}$
to $\Xx$ is \textbf{very free} if the pullback of $T_{\Xx}$ is ample:
this is equivalent to unobstructedness of morphisms relative to an
arbitrary effective divisor in $\PP^1_{\fF}$ for all morphisms
obtained by precomposing with a self-map of sufficiently positive
degree.  A smooth projective variety is \textbf{separably rationally
  connected} if and only if it has a very free rational curve, this
implies that $\Xx$ is covered by curves on which the restriction of
$T_{\Xx}$ is ample, and these are equivalent in characteristic $0$,
\cite[Theorem IV.3.7]{K}, \cite{BMcQ}.  Finally, if a general pair of
geometric points are connected by a free rational curve (that is not
necessarily very free), the variety is \textbf{freely rationally
  connected}, cf. \cite{Shen}.  In characteristic $0$, this is
equivalent to being separably rationally connected, but they may
differ in positive characteristic.  In characteristic $0$, Fano
manifolds are separably rationally connected, \cite{KMM92c}.

\mni
Separable rational connectedness implies weak approximation at places
of (strong) potentially good reduction by two theorems of the second
and third authors.

\begin{thm}\cite[Theorem 1.3]{TZWA} \label{thm-TZWA} \marpar{thm-TZWA}
  For every field $k=\ol{k}$, for every cyclic
  group $\Gamma$ of order prime to $\text{char}(k)$, every
  smooth, projective $k$-scheme with a $\Gamma$-action
  is $\Gamma$-separably rationally connected if and
  only if it is separably rationally connected.
\end{thm}

\mni
A smooth, projective, connected $k$-scheme $Y$ with an action of
$\Gamma$ by $k$-isomorphisms is $\Gamma$-\textbf{separably rationally
  connected} if $Y^{\Gamma}$ is nonempty and for every $k$-point
$(y_0,y_\infty)\in Y^{\Gamma} \times_{\SP k} Y^{\Gamma}$, there exists
a $\Gamma$-equivariant, very free $k$-morphism,
$$
u:(\PP^1_k,0,\infty) \to (Y,y_0,y_\infty),
$$
where $\Gamma$ acts faithfully on $\PP^1_k$ by scaling by roots of
unity.

\begin{thm}\cite[Theorem 1.5]{TZWA} \label{thm-TZWA2}
  \marpar{thm-TZWA2}
  For every smooth, affine, connected curve $B$ over a field
  $k=\ol{k}$, for every flat, projective $B$-scheme, weak
  approximation holds at all places of (strong) potentially good
  reduction such that the (base change) smooth fiber is separably rationally
  connected.
\end{thm}

\begin{rmk} \label{rmk-TZWA2} \marpar{rmk-TZWA2}
  In \cite[Theorem 1.5]{TZWA} there is a hypothesis that $k$ has
  characteristic $0$.  However, this hypothesis is only used to
  \emph{deduce} that the closed fiber is separably rationally
  connected: every smooth specialization of separably rationally
  connected varieties in characteristic $0$ is separably rationally
  connected.  The explicit tree of rational curves and infinitesimal
  deformation theory from \cite[Section 4]{TZWA} is valid in arbitrary
  characteristic.  Also, the alternative strategy of Roth-Starr works
  in arbitrary characteristic.
\end{rmk}

\begin{prob} \label{prob-FRC} \marpar{prob-FRC}
  For an $R$-smooth projective scheme $\Xx_R$ whose geometric generic
  fiber is separably rationally connected and whose closed fiber is
  freely rationally connected, is the closed fiber separably
  rationally connected?  Is the closed fiber equivariantly rationally
  connected for the action of a tame cyclic group $\Gamma$?
\end{prob}

\mni
Using Theorems \ref{thm-TZWA} and \ref{thm-TZWA2}, weak approximation
at places of (strong) potentially good reduction as complete
intersections follows from separable rational connectedness of
complete intersections.  Let $R$ be a DVR whose residue field
characteristic equals $p$.  Let $\Xx_R$ be an $R$-smooth, projective
scheme whose generic fiber is a complete intersection in $\PP^n$ of
type $(d_1,\dots,d_c)$.

\begin{thm}(Separable Rational Connectedness for
  Complete Intersections) \label{thm-main2} \marpar{thm-main2}
  If the Fano index is $\geq 2$, and if
  $p>\max(d_1,\dots,d_c)$, then the closed fiber is separably uniruled
  by free lines and freely rationally connected.  It is separably
  rationally connected if it is a complete intersection.
\end{thm}

\begin{rmk} \label{rmk-AIM}
  There is a related result proved by the ``rational curves working
  group'' of the AIM Workshop \emph{Rational Subvarieties in Positive
    Characteristic}, 2016.  Among other results, this working group
  proved that for every smooth, degree $d$ hypersurface
  $\Xx_k \subset \PP^n_k$ with $i\geq 2$, if $p>(d!)((d!)-1)^{n-d-1}$,
  then every line in $\Xx_k$ containing a sufficiently general point
  is a free line.  This is relevant to positive characteristic
  extensions of the Debarre -- de Jong conjecture (where existence of
  free lines is not sufficient).  The method of proof of this result
  is quite different from the method of proof of Theorem
  \ref{thm-main2}, following \cite[Section 4.2]{JanGutt} rather than
  Corollary \ref{cor-freeline}.
\end{rmk}

\mni
The inequality in the theorem is nearly sharp.  For an integer $d>1$,
the $p$-\textbf{adic valuation} of $d$ is the unique integer $v\geq 0$
such that $d=p^ve$ for a $p$-prime integer $e$.  The integer $d$ is
$p$-\textbf{special}, resp. $p$-\textbf{nonspecial}, if $1\leq e < p$,
resp. if $e>p$.  An ordered tuple of positive integers
$(d_1,\dots,d_c)$ is $p$-\textbf{special} if for every $d_i\geq p$
that is $p$-divisible, resp. $p$-prime, the integer $d_i$ is
$p$-special, resp. the integer $d_i+1$ is $p$-special.  Otherwise, it
is $p$-\textbf{nonspecial}.  Let $k$ be an algebraically closed field
$k$ of characteristic $p$.

\begin{prop}(Non-existence of Free Lines) \label{prop-sharpish}
  \marpar{prop-sharpish} 
  For every $p$-nonspecial $(d_1,\dots,d_c)$ such that
  $\text{max}(d_1,\dots,d_c)\geq p$, for every $n$
  such that Fano index is $\geq 2$, there
  exists a smooth complete intersection in $\PP^n_{k}$ of type
  $(d_1,\dots,d_c)$, yet with no free lines.
\end{prop}

\begin{prob} \label{prob-sharper} \marpar{prob-sharper}
  Does the same result hold without the $p$-nonspeciality hypothesis?
\end{prob}


\subsection{Specialization of Separable Uniruledness} 
\label{subsec-spec} \marpar{subsec-spec}

\mni
Separable uniruledness in characteristic $0$ implies separable
uniruledness for smooth specializations over a field whose
characteristic is prime to an explicit integer $D$.

\begin{defn}[Indices of curves on $\Xx$] \label{defn-D} \marpar{defn-D}
  For every field $K$, for every integral projective $K$-scheme $\Xx$,
  for every $r>0$, an $r$-\textbf{uniruling} of $\Xx$ over $K$,
  $$
  (h,\pi):Y\to \Xx\times_{\SP K} M,
  $$
  is a finite morphism of $K$-varieties
  such that $\pi$ is proper and smooth with geometric fibers $\PP^1$
  and such that the morphism from the
  $r$-fold fiber product is dominant,
  $$
  h^{(r)}:Y\times_M \dots \times_M Y \to X^r, \ \ \text{pr}_i\circ
  h^{(r)} = h\circ \text{pr}_i.
  $$
  Following \cite[Definition IV.1.7.3]{K}, the $r$-\textbf{uniruling
    index}, $u_r(K,\Xx)$, is the greatest common divisor of the
  degrees of $h^{(r)}$ for all $r$-unirulings with $h^{(r)}$
  generically finite (or $0$ if there are no $r$-unirulings).  Note
  that $u_1(K,\Xx)$ is positive if and only if $\Xx_{\ol{K}}$ is
  uniruled.  For every $r>1$, also $u_r(K,\Xx)$ is positive if and
  only if $\Xx_{\ol{K}}$ is rationally connected.

  \noindent
  When $\Xx_{\ol{K}}$ is rationally connected, then for every field
  extension $L/K$, the kernel of the following degree homomorphism is
  torsion,
  $$
  \text{deg}_{L,\Xx}:\text{CH}_0(\Xx_L) \to \ZZ.
  $$
  Following \cite{ChatzLevine}, the \textbf{torsion order},
  $\text{Tor}(K,\Xx)$, equals the exponent of this torsion Abelian
  group for $L=K(X)$.

  \noindent
  Finally, for every $m\geq 0$, the $m$-\textbf{cycle index}
  equals the positive integer $i_m(K,\Xx)$ whose ideal in $\ZZ$ is
  generated by the following composite homomorphism,
  $$
  \text{CH}^m(\Xx_K)\times \text{CH}_m(\Xx_K)\xrightarrow{\cap}
  \text{CH}_0(\Xx_K) \xrightarrow{\text{deg}_{K,\Xx}} \ZZ.
  $$
\end{defn}

\begin{rmk} \label{rmk-D} \marpar{rmk-D}
  By \cite{BlochSrinivas}, the torsion order $\text{Tor}(K,\Xx)$
  divides $u_2(K,\Xx)$.  However, already for complete intersections
  in projective space of type $(d_1,\dots,d_r)$, the torsion order
  equals $\prod_i (d_i!)$, which is much smaller than $u_2(K,\Xx)$,
  \cite{ChatzLevine}, \cite[Lemma 4.3]{Rojtman}.  Also, if $\ol{K}$
  has characteristic $0$ then the following map has torsion cokernel,
  $$
  \text{CH}_1(\Xx_{\ol{K}})/\text{CH}^{\text{alg}}_1(\Xx_{\ol{K}})\to
  \text{Hom}_{\ZZ}(\text{Pic}(\Xx_{\ol{K}})/\text{Pic}^0(\Xx_{\ol{K}}),\ZZ).
  $$
  The exponent of this cokernel is better than
  $i_1(\ol{K},\Xx_{\ol{K}})$ if the Picard rank is $>1$.
\end{rmk}
  
\mni
Let $B$ be a connected, smooth, quasi-projective scheme over
$\SP \ZZ$.  For every integer $D>1$, denote by $B_D\subset B$ the open
subscheme obtained by inverting $D$, i.e., the inverse image in $B$ of
$\SP \ZZ[1/D]$.  Let $\Xx_B\to B$ be a smooth, projective morphism of
relative dimension $n$ with geometrically connected fibers.  Denote by
$\SP K\to B$ the generic point, and denote the generic fiber of
$\Xx_B$ by $\Xx=\SP K\times_B \Xx_B$.

\begin{thm}(Separable Uniruledness in Mixed
  Characteristic) \label{thm-Y2} \marpar{thm-Y2} 
  If $\Xx$ is geometrically uniruled, then for $D=u_1(K,\Xx)$, every
  fiber over $B_D$ is separably uniruled.
\end{thm}


\subsection{Stability of the Tangent Bundle and Separable
  Rational Connectedness} \label{subsec-stab} \marpar{subsec-stab}

\mni
By the following theorem of the second author, separable uniruledness
implies free rational connectedness for manifolds with cyclic Picard
group, resp. it implies separable rational connectedness for manifolds
with cyclic Picard group and \emph{stable} tangent bundle.

\mni
For a smooth $\fF$-scheme $\Xx_\fF$, for every open subscheme
$U\subset \Xx_{\fF}$ whose complement has codimension $\geq 2$, the
pushforward of each locally free $\OO_U$-module of finite rank is a
torsion-free, coherent $\OO_{\Xx}$-module.  Such torsion-free,
coherent $\OO_{\Xx}$-modules are \textbf{reflexive}.

\mni
For a smooth, projective $\fF$-scheme $\Xx_{\fF}$ whose Picard group
is generated by an ample invertible sheaf $\OO_{\Xx}(1)$, for every
reflexive $\OO_{\Xx}$-module $\mc{E}$ of rank $r>0$, the exterior
power $\bigwedge^r_{\OO_{\Xx}}\mc{E}$ equals $\OO_{\Xx}(d)$ on an open
subset whose complement has codimension $\geq 2$.  The integer $d$ is
the \textbf{degree} of $\mc{E}$.  The \textbf{slope} is the fraction
$\mu(\mc{E}) = d/r$.  Finally, $\mc{E}$ is \textbf{stable},
resp. \textbf{semistable}, if $\mu(\mc{F})<\mu(\mc{E})$, resp. if
$\mu(\mc{F})\leq \mu(\mc{E})$, for every nonzero, reflexive
$\OO_{\Xx}$-submodule $\mc{F}$ of $\mc{E}$ with rank $<r$.

\begin{thm}\cite[Theorem 5, Corollary 9]{ZTstab} \label{thm-ZTstab}
  \marpar{thm-ZTstab} 
  Over every field $k=\ol{k}$, every smooth, projective
  $k$-variety $\Xx_k$ that is separably uniruled with cyclic
  Picard group is freely rationally connected.  It is separably
  rationally connected if $T_{\Xx_k}$ is stable.  In
  particular, a smooth complete intersection in $\PP^n_k$ is separably
  rationally connected if and only if it is separably uniruled.
\end{thm}

\begin{prob} \marpar{prob-stab} \label{prob-stab}
  In characteristic $0$, does every Fano manifold with cyclic Picard
  group have a stable tangent bundle?  Does this hold in positive
  characteristic?
\end{prob}

\mni
For an algebraically closed field $\ol{K}$, there is a general result
of Miles Reid for Fano manifolds $\Xx_{\ol{K}}$ with cyclic Picard
group whose Fano index equals $1$.

\begin{thm}\cite[Theorem 3]{ReidBog} \label{thm-Reid}
  \marpar{thm-Reid}
  For $\ol{K}$ of characteristic $0$ and $\Xx_{\ol{K}}$ as above, the
  tangent bundle is stable.
\end{thm}

\mni
We prove two cases of this for specializations to positive
characteristic.  The first case uses ``prime-to-$p$ decompositions
of the diagonal'' and results of Totaro and Gounelas-Javanpeykar on de
Rham cohomology and Picard groups for specializations of Fano
manifolds, \cite{Totaro}, \cite{GouJav}.  This version applies in the
ramified case and bounds the ``bad primes'' as divisors of
$\text{Tor}(\ol{K},\Xx_{\ol{K}})\cdot i_1(\ol{K},\Xx_{\ol{K}}).$

\mni
The second case uses the method of Kato, Fontaine-Messing, and
Deligne-Illusie on degenerations of the Hodge-de Rham spectral
sequence (Fr\"{o}hlicher spectral sequence) and crystalline
cohomology, \cite{DeligneIllusie}.  This version applies only in the
unramified case, it identifies the ``bad primes'' as all primes
$p\leq \text{dim}(\Xx_{\ol{K}})$ and all prime orders of torsion
elements of $H^*(\Xx_{\ol{K}})$, and it includes a hypothesis of no
``extra'' torsion in crystalline cohomology (it is unknown whether
such specializations of Fano manifolds have ``extra'' torsion).

\mni
Let $R$ be a complete DVR whose fraction field $K$ has characteristic
$0$ and whose residue field $k$ is perfect of characteristic $p$
(e.g., the Witt vectors $W(k)$ of the residue field).  Let $\Xx_R$ be
a smooth, projective $R$-scheme of relative dimension $n$ with
connected fibers.  There is an associated specialization map of Picard
groups of geometric fibers,
$$
\text{spec}_{\Xx_R/R}:\text{Pic}(\Xx_{\ol{K}}) \to \text{Pic}(\Xx_{\ol{k}}).
$$

\begin{thm}[Stability of Tangent Bundles in Mixed
  Characteristic] \label{thm-stab} \marpar{thm-stab} 
  Assume that the geometric generic fiber $\Xx_{\ol{K}}$ has effective
  first Chern class generating $\text{Pic}(\Xx_{\ol{K}})$.  The
  geometric closed fiber has stable tangent bundle and
  $\text{spec}_{\Xx_R/R}$ is surjective if $p$ is prime to
  $\text{Tor}(\ol{K},\Xx_{\ol{K}})$ and $\text{spec}_{\Xx_R/R}$ is
  surjective, e.g., if $p$ is prime to
  $\text{Tor}(\ol{K},\Xx_{\ol{K}})\cdot i_1(\ol{K},\Xx_{\ol{K}})$.  In
  this case, $H^0(\Xx_k,\Omega^r_{\Xx_k/k})$ vanishes for all $r>0$.

  \noindent
  Also the geometric closed fiber has stable tangent bundle and
  $\text{spec}_{\Xx_R/R}$ is surjective if $p$ is larger than the
  dimension $n$, if $R$ equals the Witt vectors $W(k)$, and if the
  crystalline cohomology $W(k)$-module $H^*(\Xx/W(k),W(k))$ is
  $p$-torsion-free.  In this case, both
  $H^0(\Xx_k,\Omega^r_{\Xx_k/k})$ and $H^r(\Xx_k,\OO_{\Xx_k})$ vanish
  for all $r>0$, the components of $\text{Pic}_{\Xx_R/R}$ are finite,
  \'{e}tale $R$-schemes, and all deformations of $\Xx_k$ are
  unobstructed.
\end{thm}

\begin{cor}[Weak Approximation for Picard Rank $1$ and Fano Index $1$] \label{cor-Y2} \marpar{cor-Y2}
  In the two cases above, if $p$ is prime to
  $D=\text{Tor}(\ol{K},\Xx_{\ol{K}})\cdot
  i_1(\ol{K},\Xx_{\ol{K}})\cdot u_1(\ol{K},\Xx_{\ol{K}})$, resp. prime to
  $D=(n!)u_1(K,\Xx)$, then the fiber $\Xx_k$ is separably rationally connected.
  For every smooth, affine, connected
  $\ol{k}$-curve $B$, for every flat, projective $B$-scheme,
  weak
  approximation holds at all points of $B$ of (strong) potentially
  good reduction where the (base change) smooth fiber is
  isomorphic to such $\Xx_{\ol{k}}$.
\end{cor}

\begin{rmk} \label{rmk-Y2} \marpar{rmk-Y2}
  In the second case, since deformations of $\Xx_k$ are unobstructed,
  existence of at least one $B$-section follows from deformation to
  characteristic $0$ and \cite{GHS}.  However, weak approximation
  certainly does not follow in that way.  Moreover, in the first case,
  since the deformations may be obstructed, even existence of one
  $B$-section requires separable rational connectedness, \cite{dJS}.
\end{rmk}
 
\mni
In principle, this corollary applies to index-$1$ Fano complete
intersections in $\PP^n$, in weighted projective spaces, in
Grassmannians, and in other projective homogeneous varieties with
cyclic Picard group.  However, bounding the integer $u_1(\fF,\Xx)$
seems difficult.  In fact, it seems difficult to explicitly compute
the genus-$0$, Gromov-Witten invariant
$\langle \tau_{e-2}(\eta_{\Xx})\rangle_{0,e}$ associated to free
curves of minimal $c_1(T_{\Xx})$-degree $e\geq 2$; this Gromov-Witten
invariant is an integer, and $u_1(\fF,\Xx)$ is a divisor of this
integer.

\mni
For Fano hypersurfaces of Fano index $1$, the minimal degree $e$
equals $2$, and the Gromov-Witten invariant $\langle \eta_{\Xx}
\rangle_{0,2}$ has been computed, \cite{BH}.
Recall that
the \textbf{Catalan number} is defined by
$$
C_d = \frac{1}{d+1}\binom{2d}{d}.
$$

\begin{thm}[Weak Approximation for Index 1 Fano
  Hypersurfaces] \label{thm-index1} \marpar{thm-index1} 
  For every integer $d\geq 4$, for the unique positive integer $n=d$
  such that the Fano index equals $1$, smooth, degree-$d$
  hypersurfaces in $\PP^n_k$ are separably uniruled by conics and
  separably rationally connected if $p>d$ and $p$ does not divide
  $(d+1)C_d-2^d$.  Also, Theorem \ref{thm-WA} holds in this case.
\end{thm}

\mni
\textbf{Acknowledgments.}
We thank Shizhang Li for alerting us to the issue of torsion in the
crystalline cohomology, and we thank Bhargav Bhatt for further
discussion of this issue.  We thank Frank Gounelas and Ariyan
Javanpeykar for stimulating conversation about specializations of
Picard groups.  Z.T. is partially supported by the
program``Recruitment of global experts", and NSFC grants No. 11871155,
No. 11831013.  J.S. thanks the organizers of the Special Session on
Algebraic Geometry of the 2018 Joint International Meeting of the
Chinese Mathematical Society and the American Mathematical Society as
well as the organizers of the 2018 International Summer School on
Arithmetic Geometry where he worked with Z. T. and R. Z.


\section{Proof of Theorem \ref{thm-stab}} \label{sec-stab} 
\marpar{sec-stab}

\mni
The key step is a vanishing theorem.  As in the statement of the
theorem, let $\Xx_R$ be a smooth, projective scheme of relative
dimension $n$ over $R$, a complete mixed characteristic DVR.  Denote
by $K$, resp. by $k$, the fraction field, resp. the residue field.
Assume that $k$ is perfect of characteristic $p$.

\begin{prop}[Hodge Coniveau in Mixed Characteristic] \label{prop-stab}
    \marpar{prop-stab}
    Assume that the geometric generic fiber $\Xx_{\ol{K}}$ is
    rationally connected.  The geometric closed fiber has vanishing de
    Rham cohomology groups $H^0(\Xx_k,\Omega^r_{\Xx_k/k})$ for every
    $r>0$ if the characteristic $p$ is prime to the torsion order,
    $\text{Tor}(\ol{K},\Xx_{\ol{K}})$. If $p$ is larger than the
    dimension $n$, if $R$ equals the Witt vectors $W(k)$, and if the
    crystalline cohomology $W(k)$-module $H^*(\Xx/W(k),W(k))$ is
    $p$-torsion-free, then for every $r>0$, both
    $H^0(\Xx_k,\Omega^r_{\Xx_k/k})$ and $H^r(\Xx_k,\OO_{\Xx_k})$
    vanish.  
\end{prop}

\begin{proof}
  First assume that $p$ is prime to the torsion order.  Up to making
  the DVR bigger, assume that $k$ is algebraically closed, assume that
  there exists a $K$-point $p$ of $\Xx_K$.  Consider the $n$-cycle
  $\gamma_p := [\Delta_{\Xx}]-[\Xx\times\{p\}] \in
  \text{CH}_n(\Xx_K\times \Xx_K)$.  For the natural map
  $\SP \ol{K}(\Xx)\to \Xx$, pullback of $\gamma_p$ by the first factor
  of $\Xx\times \Xx$ gives an element in the kernel of
  $$
  \text{deg}_{\ol{K}(\Xx),\Xx}:\text{CH}_0(\Xx_{\ol{K}(\Xx)}\to \ZZ.
  $$
  By hypothesis, this element has torsion order dividing
  $N=\text{Tor}(\ol{K},\Xx_{\ol{K}})$, and in fact these are equal.
  Thus, up to making $R$ bigger, assume that there exists a
  decomposition of the diagonal in $\text{CH}_n(\Xx_K\times \Xx_K),$
  $$
  N[\Delta_{\Xx}] = N[\Xx\times\{p\}] + Z_2,
  $$
  with $Z_2$ equal to the pushforward of a cycle class from
  $D\times \Xx_K$ for some divisor $D\subset \Xx_K$.  By \cite[Lemma
  2.4]{Totaro}, also the kernel of the degree map has torsion order
  prime to $p$,
  $$
  \text{deg}_{\ell,\Xx_\ell}:\text{CH}_0(\Xx_\ell)\to \ZZ
  $$
  for every extension field $\ell/k$.  Thus, by \cite[Lemma
  2.2]{Totaro}, for every integer $r>0$, the $k$-vector space
  $H^0(\Xx_k,\Omega^r_{\Xx_k/k})$ equals zero.

\mni
Next assume that $p>n$, assume that $R$ equals the Witt vectors
$W(k)$, and assume that the crystalline cohomology $W(k)$-module
$H^*(\Xx/W(k),W(k))$ is $p$-torsion-free, i.e., it is a finite, free
$W(k)$-module.  Since the geometric generic fiber $\Xx_{\ol{K}}$ is
Fano, it is separably rationally connected.  Thus, all of the
Dolbeault cohomology groups $h^{0,q}$ vanish, as do $h^{1,0}$ and
$h^{2,0}$.  Since $n<p$, and since the crystalline cohomology is
$p$-torsion-free, also the cohomology groups
$H^0(\Xx_{\fF},\Omega^q_{\Xx/\fF})$ vanish, as do
$H^1(\Xx_{\fF},\OO_{\Xx})$ and $H^2(\Xx_{\fF},\OO_{\Xx})$,
\cite{DeligneIllusie}. 
\end{proof}

\begin{lem} \label{lem-spec} \marpar{lem-spec}
  For a smooth, projective scheme $\Xx_R$ over $R$ with connected
  geometric fibers, a complete DVR with fraction field $K$ and residue
  field $k$, the relative Picard $R$-scheme is a countable disjoint
  union of finite, \'{e}tale $R$-schemes provided
  $H^1(\Xx_k,\OO_{\Xx_k})$ and $H^2(\Xx_k,\OO_{\Xx_k})$ vanish.  If
  $\text{char}(k)$ equals $p$ and if $H^0(\Xx_k,\Omega_{\Xx_k/k})$
  vanishes, then $\text{Pic}(\Xx_{\ol{k}})$ is $p$-torsion-free.
\end{lem}

\begin{proof}
  If $H^r(\Xx_{\fF},\OO_{\Xx})$ vanishes for $r=1,2$, then every
  invertible sheaf on the geometric closed fiber deforms uniquely to
  an invertible sheaf over the geometric generic fiber by
  infinitesimal deformation theory.  Since $\Xx_R$ is proper and
  smooth over $R$, also the relative Picard scheme satisfies the
  valuative criterion of properness.  It is a countable increasing
  union of open and closed subschemes that are finite and \'{e}tale
  which are indexed by Hilbert polynomials.

\mni
For the next part, we explain the fragment of the theory of Cartier
isomorphisms that we need.  Assume that $k$ is algebraically closed of
characteristic $p$.  Let $\mcL$ be an invertible $\OO_{\Xx_k}$-module,
and let $s$ be an isomorphism of $\OO_{\Xx_k}$-modules,
$$
s:\mcL^{\otimes p} \to \OO_{\Xx_k}.
$$
Every point of $\Xx_k$ is contained in a Zariski open affine $U$ on
which there exists a trivializing isomorphism of $\OO_U$-modules,
$$
t_U:\OO_{U}\to \mcL_U.
$$
The composite $s\circ t_U^{\otimes p}$ is multiplication by a section
$f_U\in \Gm{m}(U)$.  Multiplying $t_U$ by $g_U\in \Gm{m}(U)$ modifies
$f_U$ to $f_Ug_U^p$.  In particular, the logarithmic derivative
$f_U^{-1}df_U\in \Omega_{\Xx_k/k}(U)$ is independent of the choice of
trivializations.  Thus, there exists a global section $\alpha$ of
$\Omega_{\Xx_k/k}$ such that for every $(U,t_U)$, the logarithmic
derivative $f_U^{-1}df_U$ equals $\alpha$.

\mni
If $\alpha$ equals $0$, then every $df_U$ equals $0$.  Then, \'{e}tale
locally, $f_U$ equals $g_U^p$ for some unique $g_U$.  Thus, there
exists an \'{e}tale cover $\Xx'_k\to \Xx_k$ and a unique
trivialization $t$ of $\mcL$ on $\Xx'_k$ such that
$s\circ t^{\otimes p}$ is the identity.  The uniqueness of the
trivialization, combined with \'{e}tale descent, implies that $t$
descends to a unique trivialization of $\mcL$ on $\Xx_k$ such that
$s\circ t^{\otimes p}$ is the identity.  In particular, if
$H^0(\Xx_k,\Omega_{\Xx_k/k})$ vanishes, then there is no $p$-torsion
in $\text{Pic}(\Xx_k)$.
\end{proof}

\mni
Finally, we recall some results of Gounelas and Javanpeykar.

\begin{thm}\cite[Theorem 1.2]{GouJav} \label{thm-GJ} \marpar{thm-GJ}
  With hypotheses as above, assume that $\Xx_{\ol{K}}$ is rationally
  connected.  Then for every prime $\ell\neq p$, after tensoring with
  $\ZZ_\ell$, the specialization map of Picard groups is an
  isomorphism.  Thus, the cokernel of the specialization map is a
  finite $p$-group.
\end{thm}

\begin{proof}
  This follows immediately from \cite[Theorem 1.2]{GouJav} and the
  smooth and proper base change theorems for $\ell$-adic \'{e}tale
  cohomology.
\end{proof}

\begin{cor} \label{cor-GJ} \marpar{cor-GJ}
  Assume that the geometric generic fiber $\Xx_{\ol{K}}$ has effective
  first Chern class that generates the Picard group.  The
  specialization map on Picard groups is an isomorphism if $p$ is
  prime to
  $\text{Tor}(\ol{K},\Xx_{\ol{K}})\cdot i_1(\ol{K},\Xx_{\ol{K}})$.
  Also the specialization map is an isomorphism if $p$ is larger than
  the dimension $n$, if $R$ equals the Witt vectors $W(k)$, and if the
  crystalline cohomology $W(k)$-module $H^*(\Xx/W(k),W(k))$ is
  $p$-torsion-free.  In this second case, also the obstruction group
  to infinitesimal deformations, $H^2(\Xx_k,T_{\Xx_k/k})$, vanishes.
\end{cor}

\begin{proof}
  Lemma \ref{lem-spec} immediately implies the second case.
  In this case, since the relative Picard scheme is cyclic, and
  since there is an ample invertible sheaf by hypothesis, in fact the
  dual of the relative dualizing sheaf is ample, i.e., the relative
  dualizing sheaf is ``antiample''.  By Raynaud's theorem,
  \cite[Corollarie 2.8]{DeligneIllusie}, also
  $H^{n-2}(\Xx_{\fF},\Omega^1_{\Xx/\fF}\otimes \mcL)$ vanishes for
  every antiample invertible sheaf $\mcL$.  For $\mcL$ equal to
  $\omega_{\Xx/\fF}$, Serre duality then gives vanishing of the
  obstruction group,
  $$
  H^2(\Xx_{\fF},T_{\Xx/\fF})^\vee \cong
  H^{n-2}(\Xx_{\fF},\Omega^1_{\Xx/\fF} \otimes \omega_{\Xx/\fF}).
  $$

  \mni
  Also, Lemma \ref{lem-spec} implies vanishing of the $p$-torsion
  subgroup of the Picard group in the first case.  Thus, it suffices
  to prove that the Fano index $r_k$ of $\Xx_{\ol{k}}$ equals the Fano
  index $r_K$ of $\Xx_{\ol{K}}$.  If not, then by Theorem
  \ref{thm-GJ}, we have $r_K = p^e r_k$ for some integer $e>0$,
  cf. also \cite[Theorem 1.3]{GouJav}.  By specializing one-cycles
  from $\Xx_K$ to $\Xx_k$ and using constancy of intersection products
  (i.e., of Euler characteristics of invertible sheaves on flat,
  proper curves over a DVR), then $p$ divides the $1$-cycle index
  $i_1(\ol{K},\Xx_{\ol{K}})$, contradicting the hypothesis.
\end{proof}

\mni
The remainder of the proof follows closely the proof of \cite[Theorem
3]{ReidBog}.  By the above, we can assume that
$H^0(\Xx_{\ol{k}},\Omega^r_{\Xx_{\ol{k}}/\ol{k}})$ vanishes for every
$r>0$, and we can assume that the specialization map on Picard groups
is an isomorphism.  In the remainder of the proof, replace $\fF$ by
$\ol{k}$, i.e., assume that $\fF$ is algebraically closed.

\mni
Let $\mc{F}$ be a nonzero, reflexive $\OO_{\Xx_{\fF}}$-submodule of
$\Omega^1_{\Xx/\fF}$ of rank $q<n$.  Then
$\bigwedge^q_{\OO_{\Xx}}\mc{F}$ is an $\OO_{\Xx_{\fF}}$-submodule of
$\Omega^q_{\Xx/\fF}$ of rank $1$.  Since $\Xx_{\fF}$ is $\fF$-smooth,
also $\Omega^1_{\Xx/\fF}$ is locally free, hence $\Omega^q_{\Xx/\fF}$
is locally free.  Thus, this rank $1$ submodule factors through its
determinant.

\mni
The geometric Picard group is generated by the class
$H_{\fF}=c_1(T_{\Xx/\fF})$.  Thus, the determinant equals
$\OO_{\Xx_{\fF}}(eH_{\fF})$ for some integer $e$.  If $e\geq 0$, then
the determinant is effective, so that also
$H^0(\Xx_{\fF},\Omega^1_{\Xx/\fF})$ is nonzero, contrary to the
previous paragraph.  Thus, $e$ is negative.  So the slope
$\mu(\mc{F})$ is $\leq -1/q$, and this is strictly less than
$-1/n = \mu(\Omega^1_{\Xx/\fF})$.  Thus, $\Omega^1_{\Xx/\fF}$ is
stable.  Taking duals, also $T_{\Xx}$ is stable.


\section{Proof of Theorem \ref{thm-Y2}} \label{sec-Y2} 
\marpar{sec-Y2}

\mni
Let $R$ be a DVR with residue field $\fF$.  Let $h_R:Y_R\to \Xx_R$ be
a surjective, projective, generically finite morphism of flat,
finitely presented $R$-schemes.

\begin{lem} \label{lem-fflat} \marpar{lem-fflat}
  If $\Xx_{\fF}$ is integral, then there exists a dense open subset $W$ of
  $\Xx_R$ containing the generic point $\eta$ of $\Xx_{\fF}$ such that
  $W$ is regular, such that $h_R^{-1}(W)$ is Cohen-Macaulay, and such
  that $h_R^{-1}(W)\to W$ is finite and flat.
\end{lem}

\begin{proof}
  Since $R$ is a DVR, and since $\Xx_R$ is a finite type $R$-scheme,
  the regular locus $\Xx_R^{\text{reg}}$ is an open subscheme of
  $\Xx_R$, \cite[Corollaire 6.12.6]{EGA4}.  Since $\Xx_R$ is $R$-flat,
  also the intersection of $\Xx_R^{\text{reg}}$ with $\Xx_{\fF}$
  equals the regular locus $\Xx_{\fF}^{\text{reg}}$ of $\Xx_{\fF}$,
  \cite[Proposition 6.5.1]{EGA4}.  Since $\Xx_\fF$ is integral, the
  regular locus $\Xx_{\fF}^{\text{reg}}$ contains the generic point
  $\eta= \SP \fF(\Xx_{\fF})$.  Thus, without loss of generality,
  replace $\Xx_{\fF}$ by $\Xx_{\fF}^{\text{reg}}$, and assume that
  $\Xx_{\fF}$ is regular.  Up to replacing $\Xx_{R}$ by the unique
  connected component containing $\eta$, also assume that $\Xx_{R}$ is
  integral.  Then $\Xx_R$ is $R$-flat of constant fiber dimension $d$.

\mni
By Chevalley's theorem, there is a maximal, open subscheme $U$ of
$\Xx_R$ over which $h_R$ is finite.  Since $Y_R$ is $R$-flat, every
associated point of $Y_R$ maps to the generic point of $\SP R$.  Since
$h_R$ is generically finite, this associated point also maps to $U$.
If the image does not equal the generic point of $\Xx_R$, then the
closure of the image is an $R$-flat, proper closed subset of $\Xx_R$,
hence it does not contain the generic point $\eta$ of $\Xx_{\fF}$.  Up
to replacing $\Xx_R$ by the open complement of such proper images,
assume that every associated point of $Y_R$ maps to the generic point
of $\Xx_R$, e.g., every irreducible component $Y_{R,i}$ surjects to
$\Xx_R$.  Since $h_R$ is generically finite, also $Y_{R,i}$ is
$R$-flat of constant fiber dimension $d$.  By Krull's Hauptidealsatz,
both $Y_{\fF}$ and $\Xx_{\fF}$ have pure dimension $d$.  Thus, for
every irreducible component $Y_{\fF,i}$ of $Y_{\fF}$ on which
$h_{\fF}$ has positive fiber dimension $\geq 1$, the closed image
$h_{\fF}(Y_{\fF,i})$ has dimension $\leq d-1$.  Since $\Xx_{\fF}$ is
integral of dimension $d$, this closed image is a proper closed
subset.  Thus, up to shrinking $\Xx_R$ once more, assume that
$h_{\fF}$ is also generically finite.  In other words, $U$ contains
$\eta$.  Thus, up to replacing $\Xx_R$ by $U$, assume that $h_R$ is
finite.

\mni
By Auslander's theorem, \cite[Proposition 6.11.2(i), Corollaire
6.11.3]{EGA4}, the Cohen-Macaulay locus $Y_R^{\text{CM}}$ of $Y_R$ is
an open subscheme of $Y_R$.  Since $Y_R$ is $R$-flat, and since $R$ is
Cohen-Macaulay, the intersection of $Y_R^{\text{CM}}$ with $Y_{\fF}$
equals the Cohen-Macaulay locus $Y_{\fF}^{\text{CM}}$,
\cite[Corollaire 6.3.5(ii)]{EGA4}.  By the Generic Flatness Theorem,
there exists a dense open subset $V$ of $\Xx_{\fF}$ over which
$h_{\fF}$ is flat; denote the closed complement of $V$ by $C$.  Since
$Y_{\fF}$ has pure dimension $d$, and since $h_{\fF}$ is finite, also
$h_{\fF}^{-1}(V)$ is a dense open in $Y_{\fF}$.  Since this dense open
is flat over a regular scheme, it is Cohen-Macaulay, \cite[Proposition
6.1.5]{EGA4}.  Thus, $Y_{\fF}^{\text{CM}}$ is dense in $Y_{\fF}$.  Up
to replacing $\Xx_R$ by the open complement of $C$, assume that
$Y_R^{\text{CM}}$ contains $Y_{\fF}$.  Then the closed complement $D$
of $Y_R^{\text{CM}}$ in $Y_R$ maps finitely to a closed subset of
$\Xx_R$ whose intersection with $\Xx_{\fF}$ is empty.  Thus, up to
replacing $\Xx_R$ by the open complement of the closed image of $D$,
assume that $Y_R^{\text{CM}}$ equals $Y_R$, i.e., $Y_R$ is
Cohen-Macaulay.

\mni
Finally, since $Y_R$ is Cohen-Macaulay, since $\Xx_R$ is regular, and
since $h_R$ is finite and dominant, also $h_R$ is finite and flat,
\cite[Proposition 6.1.5]{EGA4}.
\end{proof}

\begin{prop} \label{prop-pprime} \marpar{prop-pprime}
  With hypotheses as in the previous lemma, if the generic degree of
  $h_R$ is prime to $\text{char}(\fF)$, then there exists an
  irreducible component $Y_{\fF,i}$ of $Y_{\fF}$, with its reduced
  structure, such that $Y_{\fF,i} \to \Xx_{\fF}$ is surjective,
  generically \'{e}tale and tame, i.e., the finite degree is prime to
  $\text{char}(\fF)$.
\end{prop}

\begin{proof}
  By the previous lemma, up to shrinking $\Xx_R$, assume that $\Xx_R$
  is regular and that $h_R$ is finite and flat.  Denote by
  $\text{deg}(h_R)$ the generic degree of $h_R$.

\mni
The fiber $Y_{\eta}$ of $h_R$ over $\eta$ is a finite
$\fF(\Xx_\fF)$-scheme whose length as an $\fF(\Xx_\fF)$-module equals
$\text{deg}(h_R)$.  This length equals the sum over all points
$y\in Y_{\eta}$ of the product of the length at $y$ of $Y_{\eta}$ by
the degree $[\kappa(y):\fF(\Xx_\fF)]$.  Since this sum is prime to
$p$, there exists at least one point $y$ such that
$[\kappa(y):\fF(\Xx_{\fF})]$ is prime to $p$.  In particular, since
the degree of this finite field extension is prime to $p$, this finite
field extension is separable.  Thus the closure $Y_{\fF,i}$ of this
point in $Y_{\fF}$ is an irreducible component such that the induced
morphism to $\Xx_{\fF}$ is dominant and generically \'{e}tale.
\end{proof}

\mni
Now let $M_R$ be an integral, flat, projective $R$-scheme, and let
$$
\pi_R:Y_R\to M_R, \ \ h_R:Y_R\to \Xx_R,
$$
be projective $R$-morphisms such that $\pi$ is flat with geometric
generic fiber isomorphic to $\PP^1$, and with $h_R$ surjective and
generically finite of degree $d$.

\begin{cor} \label{cor-freeline} \marpar{cor-freeline}
  If $d$ is prime to $\text{char}(\fF)$, then for at least one generic
  point $\eta$ of the closed fiber $M_{\fF}$, at least one component
  of the fiber $Y_\eta$ is a free curve in $\Xx_\fF$.  In particular,
  if the geometric generic fiber of $\pi_R$ gives a free line in
  $\Xx_K$, then also the closed fiber $\Xx_{\fF}$ contains free lines.
\end{cor}

\begin{proof}
  By the previous proposition, there exists an integral closed
  subscheme $Y_{\fF,i}$ of $Y_{\fF}$ such that the restriction,
  $$
  h_{\fF,i}:Y_{\fF,i}\to \Xx_{\fF},
  $$
  is surjective, generically \'{e}tale, and tame.  Since $h_R$ is a
  generically finite morphism between integral schemes, this closed
  subscheme contains a nonempty open $U$ subscheme of $Y_{\fF}$.  Up
  to shrinking, assume that $h_{\fF,i}$ is \'{e}tale on $U$.  Then $U$
  is $\fF$-smooth.  Since $\pi_{\fF}$ is flat, the image $V$ of $U$ in
  $M_{\fF}$ is an integral, smooth $\fF$-scheme that is an open
  subscheme of $M_{\fF}$.  Thus, the closure $M_{\fF,i}$ of $V$ is an
  irreducible component of $M_{\fF}$.

\mni
For the generic point $\eta$ of $M_{\fF,i}$, the closed subscheme
$Y_{\fF,i}$ is the closure of an irreducible component
$Y_{\fF,i,\eta}$ of the $\pi$-fiber over $\eta$.  Finally, since
$h_{\fF,i}$ is \'{e}tale, the curve $Y_{\fF,i,\eta}$ is a free
rational curve in $\Xx_{\fF}$.

\mni
As an irreducible component of the full fiber, the degree of
$Y_{\fF,i}$ is at most the degree of the generic fiber of $\pi$.
Thus, if the generic fiber is a line, also $Y_{\fF,i}$ is a line.
\end{proof}

\begin{proof}[Proof of Theorem \ref{thm-Y2}]
  Every point $b$ of $B_D$ maps to a point of $\SP \ZZ[1/D]$, either
  $\SP \QQ$ of $\SP (\ZZ/p\ZZ)$ for a prime integer $p$ not dividing
  $D$.

\mni
Consider first the case of $\SP (\ZZ/p\ZZ)$.  The closure $C$ of
$\{ b \}$ in $B\times \SP (\ZZ/p\ZZ)$ is an integral, quasi-projective
scheme over $\ZZ/p\ZZ$.  Since $\ZZ/p\ZZ$ is perfect, this scheme is
generically smooth.  Thus, after replacing $B$ by a dense open
subscheme that contains $b$, assume that $b$ is the generic point of
an integral, closed subscheme $C$ of $B\times \SP(\ZZ/p\ZZ)$ that is
smooth over $\SP (\ZZ/p\ZZ)$.

\mni
Since $C$ and $B\times \SP (\ZZ/ p\ZZ)$ are smooth over
$\SP (\ZZ/p\ZZ)$, the closed immersion of $C$ in $B$ is a regular
embedding of some codimension, say $b$.  Also $\SP \ZZ/p\ZZ$ is itself
a regular embedding into $\SP \ZZ$ of codimension $1$.  Thus, $C$ is a
regular embedding into $B_D$ of codimension $b+1$.  Denote the blowing
up along $C$ by
$$
\nu:\wt{B}_D\to B_D.
$$
This is a regular scheme that is flat over $\SP \ZZ[1/D]$.  Moreover,
the exceptional divisor $E=\nu^{-1}(C)$ is a (Zariski) locally trivial
projective space bundle of relative dimension $b$ over $C$.

\mni
Consider the pullback $\Xx_E$ of $\Xx_M$.  Assume that there exists a
pair of $\SP (\ZZ/p\ZZ)$-morphisms
$$
h_E:Y_E\to \Xx_E, \ \ \pi_E:Y_E\to M_E,
$$
such that $h_E$ is dominant, generically \'{e}tale, and tame over
$\Xx_E$, and such that $\pi_E$ is projective and flat with geometric
generic fiber isomorphic to $\PP^1$.  For the maximal open subscheme
$M_E^o$ over which $\pi_E$ is smooth, the inverse image $Y_E^o$ is a
dense open subscheme.

\mni
By Lemma \ref{lem-fflat}, there is a dense open subscheme of $\Xx_E$
over which $Y_E$ is flat.  Similarly, since the smooth locus is open,
there is a dense open subscheme such that the inverse image in $Y_E$
is contained in the dense open $Y_E^o$, and such that $h_E$ is
\'{e}tale on the inverse image open.  Since $\Xx_E$ is flat over $E$,
the image of this dense open subscheme contains a dense open subscheme
of $E$.

\mni
Since $E$ is a projective space bundle over $C$, there exists a
rational section $C\dashrightarrow E$ whose image intersects this
dense open subscheme.  Shrink $B_D$ further so that this rational
section is regular on $C$.  The pullback of $Y_E\to E$ by this section
gives a diagram,
$$
h_C:Y_C\to C, \ \ \pi_C:Y_C\to M_C
$$
such that $h_C$ is dominant, generically \'{e}tale, and tame over $C$.
Thus, to prove the existence of $Y_C$, it suffices to find $Y_E$,
i.e., it suffices to prove the result after replacing $B_D$ by
$\wt{B}_D$ and after replacing $C$ by $E$.

\mni
Denote by $R$ the local ring of $\wt{B}_D$ at the generic point $\eta$
of $E$, and denote by $\Xx_R$ the pullback of $\Xx_B$.  Since $C$ is
smooth over $\SP (\ZZ/p\ZZ)$ and since $B_D$ is smooth over $\SP \ZZ$,
also $pR$ equals the maximal ideal of $R$, and the residue field
equals the function field of $E$.  The fraction field of $R$ equals
$K$, the function field of $B_D$, and the generic fiber of $\Xx_R$
equals the generic fiber $\Xx$.

\mni
By hypothesis, $p$ is prime to $u_1(K,\Xx)$.  By the definition of
$u_1(\fF,\Xx)$, there exists a dominant, generically finite morphism
of degree $d$ prime to $p$,
$$
h^o_K:Y^o_K\to \Xx_K,
$$
and there exists a proper, smooth $\fF$-morphism of relative dimension
$1$,
$$
\pi^o_K:Y^o_K\to M^o_K,
$$
whose geometric generic fiber is isomorphic to $\PP^1$.  Up to
shrinking $M_K^o$, assume that $\pi^o_K$ is smooth.  Then the relative
dualizing sheaf $\omega_{\pi}$ has dual $\omega_{\pi}^\vee$ that is
$\pi_K^o$-very ample with vanishing higher direct images, and with
pushforward compatible with arbitrary base change.  Up to shrinking
$M^o_K$ further, assume that the pushforward is free of rank $3$.
Then $Y^o_K$ is isomorphic to a $M_K^o$-flat Cartier divisor in
$M^o_K\times_{\SP K} \PP^2_K$ of relative degree $2$.  The graph of
$h_K^o$ gives a closed immersion of $Y^o_K$ in
$M^o_K\times_{\SP K}\PP^2_K \times_{\SP K}\Xx$.

\mni
By Nagata compactification, there exists a flat, projective scheme
$M'_R$ over $R$ whose generic fiber contains $M^o_K$ as a dense open.
Denote by $Y'_R$ the closure in
$M'_R\times_{\SP R} \PP^2_R\times_{\SP R} \Xx_R$ of $Y^o_K$.  This is
flat over a dense open subscheme of $M'_R$ that contains $M^o_K$.
There is a blowing up $M_R\to M'_R$ such that the strict transform
$Y^{\text{pre}}_R$ of $Y'_R$ is an $M_R$-flat closed subscheme of
$M_R\times_{\SP R}\PP^2_R\times_{\SP R}\Xx_R$ (e.g., take $M_R$ to be
the closure of the graph of the induced $R$-rational transformation
from $M_R$ to the $R$-relative Hilbert scheme of
$\PP^2_R\times_{\SP R}\Xx_R$).  By Corollary \ref{cor-freeline}, there
exists a free curve in the closed fiber $\Xx_{\fF}$.  By standard
limit arguments, for a dense open subscheme of $E$, there is a family
of free curves in the fibers of $\Xx_E\to E$ over this dense open
subscheme.

\mni
The proof in case $b$ equals $\SP \QQ$ is similar and easier.
\end{proof}

\begin{proof}[Proof of Corollary \ref{cor-Y2}]
  The hypotheses of Theorem \ref{thm-stab} all apply.  Thus, the
  tangent bundle of $\Xx_{\fF}$ is stable.  By the previous proof,
  also $\Xx_{\fF}$ contains free rational curves.  By Theorem
  \ref{thm-ZTstab}, the closed fiber $\Xx_{\fF}$ is separably
  rationally connected.  Thus, for a fibration over a smooth, affine,
  connected $k$-curve $B$ such that at least one $k$-fiber is
  $k$-isomorphic to $\Xx_{\fF}$ as above, there exists a $B$-section
  by \cite{dJS}.  Finally, by \cite{TZWA}, weak approximation holds at
  every $k$-point of $B$ such that the fiber is $k$-isomorphic to such
  $\Xx_{\fF}$ and at places of (strong) potentially good reduction
  where the base change fiber is $k$-isomorphic to such $\Xx_{\fF}$.
\end{proof}


\section{Free Lines on Complete Intersections} \label{sec-free}
\marpar{sec-free}

\mni
Let $R$ be a commutative ring with $1$.  Let $\Xx_R$ be a projective
$R$-scheme; for simplicity, assume that $\Xx_R$ is $R$-flat.  Let
$\OO(1)$ be an $R$-ample invertible sheaf on $\Xx_R$. Let $g$, $n$,
and $e$ be nonnegative integers.  For every $R$-scheme $T$, a
\textbf{genus-$g$, $n$-pointed stable map to $\Xx_R$ of
  $\OO(1)$-degree $e$} over $T$ is a datum
$$
\zeta = (\pi:\Cc\to T, (\sigma_i:T\to \Cc)_{i=1,\dots,n}, u:\Cc \to \Xx_R)
$$
of a proper, flat morphism $\pi$ of relative dimension $1$ whose
geometric fibers are connected, reduced, at-worst-nodal curves, of an
ordered $n$-tuple $(\sigma_i)$ of pairwise disjoint sections of $\pi$
with image in the smooth locus of $\pi$, and of an $R$-morphism $u$
such that the log relative dualizing sheaf,
$$
\omega_{\pi,\sigma}:=\omega_\pi\lt( \sum_{i=1}^n \ul{\sigma_i(T)} \rt)
$$
is $u$-ample (\textbf{stability}) and such that $u^*\OO(1)$ has
relative degree $e$ over $T$. For a genus-$g$, $n$-pointed stable map
to $\Xx_R$ of $\OO(1)$-degree $e$,
$$
\wt{\zeta} = (\wt{\pi}:\wt{\Cc}\to T, (\wt{\sigma_i}:T\to
\wt{\Cc})_{i=1,\dots,n}, \wt{u}:\wh{\Cc} \to \Xx_R), 
$$
a $2$-\textbf{morphism} from $\zeta$ to $\wt{\zeta}$ is a
$T$-isomorphism $\phi:\Cc \to \wt{\Cc}$ such that $\wt{u}\circ \phi$
equals $u$ and such that $\phi\circ \sigma_i$ equals $\wt{\sigma}_i$
for every $i=1,\dots,n$.  For every $R$-morphism $f:T'\to T$ the
$f$-\textbf{pullback} of $\zeta$ is
$$
f^*\zeta = \lt( \text{pr}_{T'}:T'\times_T \Cc \to T',
((\text{Id}_{T'},\sigma_i):T'\to T'\times_T \Cc)_{i=1,\dots,n}, u\circ
\text{pr}_{\Cc} : T'\times_T \Cc \to \Xx_R \rt).
$$
Altogether, these operations define a stack in groupoids
$\Kgnb{g,n}((\Xx_R/R,\OO(1)),e)$ over the category of $R$-schemes.

\begin{thm}\cite[Theorem 2.8]{AbramOort} \label{thm-alg} \marpar{thm-alg}
  The stack in groupoids $\Kgnb{g,n}((\Xx_R/R,\OO(1)),e)$ is an
  algebraic (Artin) stack with finite diagonal that is $R$-proper.
  There is a coarse moduli space,
  $$
  f_{g,n,e}:\Kgnb{g,n}((\Xx_R/R,\OO(1)),e) \to \kgnb{g,n}((\Xx_R/R,\OO(1)),e),
  $$
  and $\kgnb{g,n}((\Xx_R/R,\OO(1)),e)$ is a projective $R$-scheme.
  The maximal open subscheme on which $f$ is an isomorphism
  parameterizes stable maps that have only the identity automorphism.
\end{thm}

\begin{cor} \label{cor-alg} \marpar{cor-alg}
  For $g=0$, for $e=1$, and for every integer $n\geq 0$, the morphism
  $f_{0,n,1}$ is an isomorphism, i.e.,
  $\Kgnb{0,n}((\Xx_R/R,\OO(1)),e)$ is a projective $R$-scheme.
\end{cor}

\begin{proof}
  This is well-known, but we include a proof for completeness.  It
  suffices to check for every algebraically closed field, $R\to k$,
  every stable map defined over $k$ has trivial automorphism group.
  Consider first the case that $n$ equals $0$.  For a genus-$0$ stable
  map $u:C\to \Xx_k$ of degree $1$, there is at most one irreducible
  component on which $u^*\OO(1)$ has positive degree.  Since every
  other component must be contracted by $u$, the stability hypothesis
  implies that $C$ is irreducible.  Since $C$ is a connected,
  at-worst-nodal curve of arithmetic genus $0$, in fact $C$ is smooth.
  For every sufficiently positive integer $d$, the invertible sheaf
  $\OO(d)$ is very ample on $\Xx$, so that the stable map gives a
  morphism $C\to \PP^{N_d}$ of degree $d$.  The automorphism of every
  such morphism is cyclic of order dividing $d$.  Choosing $d$ among a
  set of sufficiently positive, relatively prime integers, the
  automorphism group of the stable map is trivial.

\mni
For a stable map with $n>0$,
$$
(C,(q_1,\dots,q_n),u:C\to \Xx_k),
$$
there is a unique irreducible component $C_i$ of $u$ such that
$u^*\OO(1)$ has degree $1$, and every automorphism acts as the
identity on $C_i$ by the previous paragraph.  Let
$q_{n+1},q_{n+2},q_{n+3} \in C_i$ be distinct $k$-points that are
different from the finitely many nodes and marked points that are
contained in $C_i$.  Then every automorphism of the stable map is, in
particular, an automorphism of the marked curve
$(C,(q_1,\dots,q_n,q_{n+1},q_{n+2},q_{n+3}))$.  The stability
condition for stable maps implies that this $n+3$-pointed curve is
also stable.  Since the stack $\Kgnb{0,n+3}$ is a projective scheme,
the claim follows.
\end{proof}

\mni
For each finite sequence of positive integers, $(d_1,\dots,d_c)$, let
$D$ be the product of $d_i!$ for $i=1,\dots,c$.  Let $n$ be the unique
integer such that the Fano index equals $2$, i.e.,
$n=1+(d_1+\dots+d_c)$.  Let $B'$ be the affine space over $\SP \ZZ$
that parameterizes ordered $c$-tuples $(g_1,\dots,g_c)$ of homogeneous
polynomials $g_i$ in variables $(t_0,\dots,t_n)$ such that
$\text{deg}(g_i) = d_i$.  Let $B\subset B'$ be the dense open
subscheme that parameterizes such $c$-tuples whose Jacobian ideal has
empty zero scheme, i.e., such that $\text{Zero}(g_1,\dots,g_c)$ is a
smooth complete intersection of type $(d_1,\dots,d_c)$ in $\PP^n$.
Let $\Xx_B\subset \PP^n_B$ be the universal smooth, complete
intersection of type $(d_1,\dots,d_c)$ and Fano index $2$.

\begin{prop} \label{prop-Y} \marpar{prop-Y}
  For the $\ZZ[1/D]$-scheme $B_D$, every point parameterizes a smooth
  complete intersection $\Xx$ in $\PP^n$ of type $(d_1,\dots,d_c)$ and
  Fano index $2$ such that there exists an integral closed subscheme
  $Y$ of the space $\Kgnb{0,1}((\Xx,\OO(1)),1)$ of pointed lines in
  $\Xx$ with $Y\to \Xx$ surjective, generically \'{e}tale and tame.
\end{prop}

\begin{proof}
  By the proof of Theorem \ref{thm-Y2}, it suffices to prove that the
  evaluation morphism for $1$-pointed lines is dominant and
  generically finite of degree $D$ for the generic complete
  intersection in characteristic $0$.

\mni
For the function field $K=\QQ(M_{\QQ})$ and the generic complete
intersection $\Xx_K\subset \PP^n_K$ of type $(d_1,\dots,d_c)$, the
tangent direction of each line at the marked point defines a closed
immersion
$$
\tau:\Kgnb{0,1}((\Xx_K/K,\OO(1)),1) \hookrightarrow
\PP_{\Xx_K}(T_{X_K/K}),
$$
whose image is, relative to the projection to $\Xx_K$, generically a
complete intersection of type
$(2,3,\dots,d_1, 2,3,\dots,d_2,\dots,2,3,\dots, d_c)$,
cf. \cite[Exercise V.4.10.5, p. 272]{K}, \cite[Lemma 4.3]{Rojtman}.
In particular, the codimension of this complete intersection equals
$(d_1-1) + \dots + (d_c-1) = (d_1+\dots+d_c) - c$.  Of course the
projective bundle of the tangent bundle has relative dimension
$n-c-1$.  Thus, since $n$ equals $1+(d_1\dots+d_c)$, this complete
intersection $Y_K$ has relative dimension $0$ over $\Xx_K$, i.e., it
is generically finite over $\Xx_K$.  Finally, by B\'{e}zout's theorem,
the relative degree of this generically finite map equals
$(d_1!)\cdots (d_c!) = D$.
\end{proof}

\begin{proof}[Proof of Theorem \ref{thm-main2}]
  By Corollary \ref{cor-freeline} to prove that $\Xx_k$ is separably
  uniruled by lines, it suffices to construct an irreducible closed
  subscheme $Y_{\text{Frac}(R)}$ of the $\text{Frac}(R)$-fiber of
  $\Kgnb{0,1}((\Xx_R/R,\OO(1)),1)$ such that the induced morphism
  $Y_{\text{Frac}(R)} \to \Xx_{\text{Frac}(R)}$ is surjective and
  generically \'{e}tale of degree prime to $p$.  This is proved by
  induction on $n$ for every integer at least
  $n_0:=1+(d_1+\dots+d_c)$.

  \mni
  The base case is when $n=n_0$, in which
  case $Y_{\text{Frac}(R)}$ exists by Proposition \ref{prop-Y}.
  By way of induction, assume that $n\geq n_0+1$ and assume that the
  result holds for $n-1$.  Denote by $K$ the function field of $\PP^1$
  over $\text{Frac}(R)$.  
  Since $\text{Frac}(R)$
  is an infinite field, by Bertini's theorem, there exists a pencil of
  hyperplane sections, $\wt{\Xx}_{\PP^1} \subset
  \Xx_{\text{Frac}(R)}\times \PP^1$, whose generic fiber $\wt{\Xx}_K$
  is
  $K$-smooth.
  Thus $\wt{\Xx}_K$ is a smooth complete intersection of type
  By the induction hypothesis, there
  exists an integral closed subscheme $\wt{Y}_K$ in
  $\Kgnb{0,1}((\wt{\Xx}_K/K,\OO(1)),1)$ such that $\wt{Y}_K\to \wt{\Xx}_K$
  is surjective and generically \'{e}tale of degree prime to $p$.
  Define $\wt{Y}_{\PP^1}$ to be the closure of $\wt{Y}_K$ in
  $\Kgnb{0,1}((\wt{\Xx}_{\PP^1}/\PP^1,\OO(1)),1)$.  By covariance of
  the stack of stable maps, there is an induced proper
  $\text{Frac}(R)$-morphism,
  $$
  \Kgnb{0,1}((\wt{\Xx}_{\PP^1}/\PP^1,\OO(1)),1) \to
  \Kgnb{0,1}(\Xx_{\text{Frac}(R)}/\text{Frac}(R)),1),
  $$
  i.e., lines in the hyperplane sections of $\Xx_{\text{Frac}(R)}$
  give lines in $\Xx_{\text{Frac}(R)}$ (that happen to be contained in
  the specified hyperplane section).  Define $Y_{\text{Frac}(R)}$ to
  be the image of $\wt{Y}_{\PP^1}$ in
  $\Kgnb{0,1}((\Xx_{\text{Frac}(R)}/\text{Frac}(R),\OO(1)),1)$.  Since
  the generic fiber of $Y_{\text{Frac}(R)} \to \Xx_{\text{Frac}(R)}$
  equals the generic fiber of $\wt{Y}_{\PP^1}\to \wt{\Xx}_{\PP^1}$,
  also $Y_{\text{Frac}(R)}$ is surjective, generically \'{e}tale, and
  of degree prime to $p$ over $\Xx_{\text{Frac}(R)}$.  Thus, by
  induction on $n$, there always exists an integral closed subscheme
  $Y_{\text{Frac}(R)}$ of
  $\Kgnb{0,1}((\Xx_{\text{Frac}(R)}/\text{Frac}(R),\OO(1)),1)$ that is
  generically \'{e}tale of degree prime to $p$ over
  $\Xx_{\text{Frac}(R)}$.

\mni
By \cite{GouJav}, the closed fiber of $\Xx_R$ has cyclic Picard group.
Since the closed fiber is separably uniruled, by \cite[Theorem 5,
Corollary 9]{ZTstab}, the closed fiber is also freely rationally
connected, and the closed fiber is even separably rationally connected
when it is a complete intersection.
\end{proof}

\begin{proof}[Proof of Theorem \ref{thm-WA}]
  By hypothesis, for a tame, Galois base change by
  $\wht{\OO}_{B,b}\to R$, for a $\Gamma$-equivariant modification, the
  closed fiber over $b$ is a smooth complete intersection in $\PP^n$
  of type $(d_1,\dots,d_c)$.  By Theorem \ref{thm-main2}, the closed
  fiber is separably rationally connected.  By Theorem
  \ref{thm-TZWA2}, the original family satisfies weak approximation at
  $b$.
\end{proof}

\begin{proof}[Proof of Theorem \ref{thm-index1}]
  We apply Corollary \ref{cor-Y2}.  By \cite{ChatzLevine} and
  \cite{Rojtman}, the torsion order $\text{Tor}(\ol{K},\Xx_{\ol{K}})$
  divides $d!$, and existence of lines on such hypersurfaces implies
  that $i_1(\ol{K},\Xx_{\ol{K}})$ equals $1$.  Thus, it suffices to
  prove that for the generic Fano hypersurface $\Xx$ in $\PP^d_K$ of
  degree $d$ and Fano index $1$ in characteristic $0$, the radical of
  the uniruling index $u_1(K,\Xx)$ divides $(d!)(d+1)C_d-2^d).$ This
  is similar to the proof of Theorem \ref{thm-main2}, except using
  pointed conics in place of pointed lines, i.e., using
  $$
  \text{ev}:\Kgnb{0,1}((\Xx_K/K,\OO(1)),2) \to \Xx_K.
  $$
  Note that this requires working in characteristic prime to $2$, but
  this is already implied by the hypothesis on the characteristic.  In
  the case of a Fano hypersurface of index $1$ in characteristic $0$,
  the generic degree $e$ of this map is computed in \cite{BH},
  $$
  e = \frac{(d!)^2}{2^{d+1}} \lt( (d+1)C_d-2^d \rt).
  $$
  As in the proof of Theorem \ref{thm-WA}, separable rational
  connectedness implies weak approximation at places of (strong)
  potentially good reduction via Theorem \ref{thm-TZWA2}.
\end{proof}


\section{Complete Intersections with No Free Lines} \label{sec-nonfree}
\marpar{sec-nonfree}

\mni
For all integers $n,d\geq $, denote by $g_{n,d}$ the degree-$d$ Fermat
homogeneous polynomial
$$
g_{n,d}(t_0,\dots,t_n) = t_0^d + \dots + t_n^d.
$$

\begin{lem} \label{lem-Fermat} \marpar{lem-Fermat}
  For every algebraically closed field $k$ of characteristic $p$ prime
  to $d$, the zero scheme $\Xx_k$ of $g_{n,d}$ in $\PP^n_k$ is a
  $k$-smooth hypersurface.  If $d+1$ is $p$-nonspecial and if
  $n\geq d$, then every irreducible component of every fiber of the
  evaluation morphism,
  $$
  \text{ev}:\Kgnb{0,1}((\Xx_k/k,\OO(1)),1) \to \Xx_k,
  $$
  has dimension strictly larger than the expected fiber dimension
  $n-d-1$.
\end{lem}

\begin{proof}
  Consider $\Kgnb{0,2}((\Xx_k/k,\OO(1)),1)$ with its natural
  evaluation morphism to
  $\PP^n_k\times \PP^n_k = \text{Proj}\ k[s_0,\dots,s_n] \times
  \text{Proj}\ k[t_0,\dots,t_n]$.  The image is the common zero locus
  of the collection of bihomogeneous polynomials
  $$
  g_{n,d,\ell}(s_0,\dots,s_n,t_0,\dots,t_n) = \sum_{j=0}^n
  \binom{d}{\ell} s_j^{d-\ell} t_j^\ell,
  \ \ 0\leq \ell \leq d.
  $$
  If any binomial coefficient $\binom{d}{\ell}$ is zero, then at least
  one $g_{n,d,\ell}$ is identically zero.  Then, by Krull's
  Hauptidealsatz, every irreducible component has dimension strictly
  larger than the expected dimension.  Since the difference of
  dimensions of domain and image is a lower bound on the dimension of
  every irreducible component of every fiber of a morphism, also every
  irreducible component of every fiber of $\text{ev}$ has dimension
  strictly larger than the expected fiber dimension.

\mni
By hypothesis, $d+1$ is nonspecial, i.e., $d+1=p^{v-1}e$ for a
$p$-prime integer $e>p$.  By the Division Algorithm, $d$ equals
$p^va+r$ for an integer $a\geq 0$ and an integer $r$ with
$0\leq r < p^v$.  Since $d+1=p^{v-1}e > p^v$, it follows that
$a\geq 1$.  Since $d$ is $p$-prime, also $r\geq 1$.  Since $p^v$ does
not divide $d+1$, also $r<p^v-1$.  Thus,
$$
(s+t)^d = (s+t)^{p^va}(s+t)^r = (s^{p^v}+t^{p^v})^a(s+t)^r =
$$
$$
\lt( \sum_{b=0}^a
\binom{a}{b} s^{p^v(a-b)}t^{p^vb} \rt)\lt( \sum_{q=0}^{r}
\binom{r}{q} s^{r-q}t^q \rt) = 
$$
$$
\sum_{q=0}^r\sum_{b=0}^{a} \binom{r}{q}\binom{a}{b}
s^{p^v(a-b) + r-q}t^{p^vb+q}.
$$
In particular, the coefficient of the monomial $s^{d-\ell}t^\ell$ is
nonzero only if $\ell$ is congruent modulo $p^v$ to $0,\dots,r$, i.e.,
only for $r+1$ of the total $p^v$ congruence classes.  Since
$r+1 < p^v$, there are some binomial coefficients that are identically
zero.
\end{proof}

\begin{prop} \label{prop-Fermat} \marpar{prop-Fermat}
  Let $k$ be an algebraically closed field of characteristic $p$.  Let
  $(d_1,\dots,d_c)$ be positive integers such that there exists at
  least one $i$ with $d_i$ a $p$-prime, $p$-nonspecial integer
  $\geq p$.  Let $n\geq 1+(d_1+\dots+d_c)$.  Then there exists a
  $k$-smooth complete intersection in $\PP^n_k$ of type
  $(d_1,\dots,d_c)$ of Fano index $\geq 2$ that has no free lines.
\end{prop}

\begin{proof}
  Choose $g_i$ to be $g_{n,d_i}$.  By the previous lemma and the
  hypothesis that $d_i\not\equiv 0(\text{mod}\ p)$, the zero scheme is
  $k$-smooth.  By Bertini's Theorem, for general choice of homogeneous
  equation $g_j$ for $j\neq i$, the zero scheme
  $\Xx_k = \text{Zero}(g_1,\dots,g_c)$ is a $k$-smooth complete
  intersection of type $(d_1,\dots,d_c)$ in $\PP^n_k$.  Yet, by the
  previous lemma, every irreducible component of every fiber of
  $$
  \text{ev}:\Kgnb{0,1}((\Xx_k/k,\OO(1)),1) \to \Xx_k
  $$
  has dimension $\geq [n-1-(d_1+\dots+d_c)]+1$.  For a pointed free
  line, the evaluation morphism is smooth with relative dimension
  equal to the expected dimension, $[n-1-(d_1+\dots+d_c)]$.  Thus,
  there exists no free line in $\Xx_k$.
\end{proof}

\mni
For every integer $q\geq 1$, for $d=qp$, for every integer $n$, denote
by $g_{n;d}$ the following degree-$d$ homogeneous polynomial,
$$
g_{n;d} = t_0t_n^{qp-1}+\sum_{j=0}^{n-1}t_j^{qp-1}t_{j+1}.
$$
If $p$ divides $n+1$, define $\wt{g}_{n;d} = g_{n;d} + t_0^{d}$.

\begin{lem} \label{lem-Fermat2} \marpar{lem-Fermat2}
  For every algebraically closed field $k$ of characteristic $p$, for
  every $p$-divisible integer $d>0$, for every integer $n$ such that
  $n+1$ is $p$-prime, resp. such that $n+1$ is $p$-divisible, the zero
  scheme $\Xx_k$ of $g_{n;d}$ is $k$-smooth, resp. the zero scheme
  $\Xx_k$ of $\wt{g}_{n;d}$ is $k$-smooth.  If the integer $d$ is not
  $p$-special, and if $n\geq d$, then every irreducible component of
  every fiber of the evaluation morphism,
  $$
  \text{ev}:\Kgnb{0,1}((\Xx_k/k,\OO(1)),1) \to \Xx_k,
  $$
  has dimension strictly larger than the expected fiber dimension
  $n-d-1$.
\end{lem}

\begin{proof}
  Consider the indices of variables as elements in $\ZZ/(n+1)\ZZ$.
  Then the partial derivative with respect to $t_i$ of $g_{n;d}$ and
  $\wt{g}_{n;d}$ equals
  $$
  \partial_{t_i} g_{n;d} = \partial_{t_i}\wt{g}_{n;d} =
  t_{i-1}^{qp-1} - t_i^{qp-2}t_{i+1}.
  $$
  In particular, multiplying by $t_i$ gives
  $$
  t_i \partial_{t_i} g_{n;d} = t_{i-1}^{qp-1}t_i - t_i^{qp-2}t_{i+1}.
  $$
  Thus, modulo the ideal of partial derivatives,
  $$
  I=\langle \partial g_{n;d} \rangle
  = \langle \partial \wt{g}_{n;d} \rangle,
  $$
  every monomial with nonzero coefficient in $g_{n;d}$ is congruent to
  a common element, say $s=t_0^{qp-1}t_1$,
  $$
  t_i^{qp-1}t_{i+1} \cong s \lt( \text{mod} I \rt), \ \ i=0,\dots,n,
  $$
  Summing over all $i$,
  $$
  g_{n;d} \cong (n+1)s \lt( \text{mod} I \rt), \ \
  g_{n;d} \cong t_0^{pq}+(n+1)s \lt( \text{mod} I \rt).
  $$
  If $n+1$ is not divisible by $p$, then for every $i$ the element
  $t_i^{qp-1}t_{i+1}$ is in the Jacobian ideal of $g_{n;d}$.  Thus the
  radical of the Jacobian ideal contains $t_i^{qp-2}t_{i+1}$.  Using
  the partial derivative above, the radical of the Jacobian ideal also
  contains $t_{i-1}^{qp-1}$, and thus $t_{i-1}$.  Since the radical of
  the Jacobian ideal contains $t_{i-1}$ for every $i$, the Jacobian
  ideal contains $\langle t_0^e,\dots,t_n^e\rangle$ for some integer
  $e\geq 1$.  Thus, when $n+1$ is not divisible by $p$, the zero locus
  of $g_{n;d}$ is smooth.

\mni
When $n+1$ is divisible by $p$, then $\wt{g}_{n;d}$ is congruent to
$t_0^{pr}$ modulo $I$.  Thus, the Jacobian ideal contains $t_0^{pr}$.
Using the partial derivative above, the radical of the Jacobian ideal
contains $t_{0}$.  Using the partial derivatives for $t_n$ and
$t_{n-1}$, also the radical contains $t_n$ and $t_{n-1}$.  Then using
the partial derivatives for $t_{n-2}$ and $t_{n-3}$, also the radical
contains $t_{n-2}$ and $t_{n-3}$, etc.  Finally, the radical of the
Jacobian ideal contains every $t_i$, $i=0,\dots,n$.  So the Jacobian
ideal contains $\langle t_0^e,\dots,t_n^e \rangle$ for some integer
$e\geq 1$.  Thus, when $n+1$ is divisible by $p$, the zero locus of
$\wt{g}_{n;d}$ is smooth.

\mni
As in the proof of Lemma \ref{lem-Fermat}, in the $k$-algebra
$k[s_0,\dots,s_n,t_0,\dots,t_n]$ write,
$$
g_{n;d}(s_0+t_0,\dots,s_n+t_n) = \sum_{\ell=0}^d
g_{n;d,\ell}(s_0,\dots,s_n,t_0,\dots,t_n), 
$$
respectively, if $p$ divides $n+1$,
$$
\wt{g}_{n;d}(s_0+t_0,\dots,s_n+t_n) = \sum_{\ell=0}^d
\wt{g}_{n;d,\ell}(s_0,\dots,s_n,t_0,\dots,t_n), 
$$
where $g_{n;d,\ell}$, resp. $\wt{g}_{n;d,\ell}$, is bihomogeneous in
$(s_0,\dots,s_n)$ and $(t_0,\dots,t_n)$ of bidegree $(\ell,d-\ell)$.
If any $g_{n;d,\ell}$ is identically zero, then by Krull's
Hauptidealsatz every irreducible component of every fiber of
$\text{ev}$ has dimension strictly larger than the expected dimension.

\mni
Now assume that the $p$-multiple, $d=pq$, is $p$-nonspecial.  Then $d$
equals $p^{v-1}e$ for a $p$-prime integer $e>p$.  By the Division
Algorithm, $d$ equals $p^va+r$ for an integer $a\geq 0$ and an integer
$r$ with $0\leq r < p^v$.  Since $d=p^{v-1}e > p^v$, it follows that
$a\geq 1$.  Since $p^v$ does not divide $d$, also $r\geq 1$.  Since
$p^{v-1}$ does divide $d$, $r$ is $\leq p^v-p^{v-1}$.  Thus,
$$
(s+t)^{d-1} = (s+t)^{p^va}(s+t)^{r-1} = (s^{p^v}+t^{p^v})^a(s+t)^{r-1} =
$$
$$
\lt( \sum_{b=0}^a
\binom{a}{b} s^{p^v(a-b)}t^{p^vb} \rt)\lt( \sum_{m=0}^{r-1}
\binom{r-1}{m} s^{r-1-m}t^m \rt) = 
$$
$$
\sum_{m=0}^{r-1}\sum_{b=0}^{a} \binom{r-1}{m}\binom{a}{b}
s^{p^v(a-b) + r-1-m}t^{p^vb+m}.
$$
In particular, the coefficient of the monomial $s^{d-1-\ell}t^\ell$ is
nonzero only if $\ell$ is congruent modulo $p^v$ to $0,\dots,r-1$.
Similarly, the summand
$$
(s_j+t_j)^{d-1}(s_{j+1}+t_{j+1}) =
(s_{j+1}+t_{j+1})\sum_{m=0}^{r-1}\sum_{b=0}^{a} \binom{r-1}{m}\binom{a}{b}
s_j^{p^v(a-b) + r-1-m}t_j^{p^vb+m},
$$
has nonzero bihomogeneous component of bidegree $(d-\ell,\ell)$ only
if $\ell$ is congruent modulo $p^v$ to $0,\dots,r$.  Summing over all
integers $j$, the bihomogeneous component $g_{n;d,\ell}$ of bidegree
$(d-\ell,\ell)$ is nonzero only if $\ell$ is congruent modulo $p^v$ to
$0,\dots, r$.  Since $r\leq p^v-p^{v-1}$, this is at most
$p^v+1-p^{v-1}$ of the total $p^v$ congruence classes.  So, when $d$
is $p$-nonspecial and $n+1$ is $p$-prime, every irreducible component
of every fiber of $\text{ev}$ has dimension strictly larger than the
expected dimension.

\mni
When $n+1$ is divisible by $p$, then $\wt{g}_{n;d}$ has one additional
monomial $t_0^d = t_0^{p^{v-1}e}$.  By the Binomial Theorem,
$$
(s_0+t_0)^{p^{v-1}e} = \sum_{b=0}^e \binom{e}{b} s_0^{d-p^{v-1}b} t_0^{p^{v-1}b}.
$$
Thus, the bihomogeneous component of bidegree $(d-\ell,\ell)$ is
nonzero only if $\ell$ is a multiple of $p^{v-1}$.  In particular, the
excluded congruence classes $p^v-p^{v-1}<\ell < p^v$ for nonzero
$g_{n;d,\ell}$ is still excluded for $\wt{g}_{n;d,\ell}$.  Every prime
$p$ is $\geq 2$.  Also, by hypothesis, $d$ is divisible by $p$, so
that $p^{v-1}$ is at least $p\geq 2$.  Thus, the number $p^{v-1}-1$ of
excluded congruence classes is $\geq 1$.  Thus, as above, every
irreducible component of every fiber of $\text{ev}$ has dimension
strictly larger than the expected dimension.
\end{proof}

\begin{prop} \label{prop-Fermat2} \marpar{prop-Fermat2}
  Let $k$ be an algebraically closed field of characteristic $p$.  Let
  $(d_1,\dots,d_c)$ be positive integers such that there exists at
  least one $i$ with $d_i$ a $p$-divisible, $p$-nonspecial integer
  $\geq p$.  Let $n\geq 1+(d_1+\dots+d_c)$.  Then there exists a
  $k$-smooth complete intersection in $\PP^n_k$ of type
  $(d_1,\dots,d_c)$ of Fano index $\geq 2$ that has no free lines.
\end{prop}

\begin{proof}
  The proof is the same as the proof of Proposition \ref{prop-Fermat},
  but using Lemma \ref{lem-Fermat2} in place of Lemma
  \ref{lem-Fermat}.
\end{proof}

\mni
Together, Propositions \ref{prop-Fermat} and \ref{prop-Fermat2}
establish Proposition \ref{prop-sharpish}.

\bibliography{my}
\bibliographystyle{alpha}

\end{document}